\documentclass[oneside]{amsart}
\usepackage{afterpage}
\usepackage{longtable}
\usepackage[cmtip,all]{xy}
\usepackage[latin1]{inputenc}
\usepackage{amssymb,amsmath}
\usepackage{verbatim}
\usepackage{color}
\usepackage{enumerate}
\usepackage{url}
\usepackage{ifthen}
\usepackage{geometry}
\usepackage[colorlinks]{hyperref}
\hypersetup{citecolor=blue}

\usepackage{array}

\usepackage{array}

\newtheorem{theorem}{Theorem}[section]
\newtheorem{lemma}[theorem]{Lemma}
\newtheorem{proposition}[theorem]{Proposition}
\newtheorem{conjecture}[theorem]{Conjecture}
\newtheorem{corollary}[theorem]{Corollary}

\theoremstyle{definition}

\theoremstyle{remark}
\newtheorem{remark}[theorem]{Remark}

\numberwithin{equation}{section}

\newcommand\Z{\ensuremath{\mathbb Z}}

\newcommand\Q{\ensuremath{\mathbb Q}}
\newcommand\R{\ensuremath{\mathbb R}}
\newcommand\C{\ensuremath{\mathbb C}}
\newcommand\F{\ensuremath{\mathbb F}}

\def\cO{{\ensuremath{\mathcal O}}}

\newcommand\Div{\operatorname{Div}}
\newcommand\End{\operatorname{End}}

\newcommand\Ind{\operatorname{Ind}}

\newcommand\M{\operatorname{M}}

\newcommand\ord{\operatorname{ord}}

\newcommand\PGL{\operatorname{PGL}}

\newcommand\SL{\operatorname{SL}}

\newcommand{\cH}{\mathcal{H}}
\newcommand{\cT}{\mathcal{T}}
\newcommand{\fp}{{\mathfrak{p}}}

\newcommand{\fm}{{\mathfrak{m}}}
\newcommand{\fn}{{\mathfrak{n}}}

\newcommand{\fq}{{\mathfrak{q}}}
\newcommand{\cU}{{\mathcal{U}}}
\newcommand{\cV}{{\mathcal{V}}}
\newcommand{\cE}{{\mathcal{E}}}
\newcommand{\cF}{{\mathcal{F}}}
\newcommand{\fl}{{\mathfrak{l}}}
\newcommand{\tto}[1]{%
\ifthenelse{\equal{#1}{}}{\to}{\stackrel{#1}{\to}}}
\newcommand{\cB}{{\mathcal{B}}}

\newcommand{\smtx}[4]{\left(\begin{smallmatrix}#1&#2\\#3&#4\end{smallmatrix}\right)}

\def\M{\operatorname{M}}

\def\P{\mathbb P}

\def\fN{\mathfrak N}

\def\ord{\operatorname{ord}}

\newcommand{\ra}{\rightarrow}

\newcommand{\lra}{\longrightarrow}

\newcommand{\QQ}{\Q}

\def\fD{\mathfrak{D}}

\def\Xint#1{\mathchoice
{\XXint\displaystyle\textstyle{#1}}%
{\XXint\textstyle\scriptstyle{#1}}%
{\XXint\scriptstyle\scriptscriptstyle{#1}}%
{\XXint\scriptscriptstyle\scriptscriptstyle{#1}}%
\!\int}
\def\XXint#1#2#3{{\setbox0=\hbox{$#1{#2#3}{\int}$}
\vcenter{\hbox{$#2#3$}}\kern-.5\wd0}}

\begin{document}

\title[Uniformization of modular elliptic curves via $p$-adic periods]{Uniformization of modular elliptic curves\\via $p$-adic periods}



\author{Xavier Guitart}
\address{Departament d'Algebra i Geometria\\
Universitat de Barcelona \\ 
Spain}
\email{xevi.guitart@gmail.com}
\urladdr{http://atlas.mat.ub.edu/personals/guitart/}
\thanks{The first named author was partially supported by SFB/TR 45. The second named author was supported by EPSRC CAF EP/J002658/1.}

\author{Marc Masdeu}
\address{Mathematics Institute \\
University of Warwick \\
United Kingdom}
\email{M.Masdeu@warwick.ac.uk}
\urladdr{http://warwick.ac.uk/mmasdeu/}

\author{Mehmet Haluk \c{S}eng\"un}

\address{School of Mathematics and Statistics\\ 
University of Sheffield \\
 United Kingdom}
\email{M.Sengun@sheffield.ac.uk}
\urladdr{http://haluksengun.staff.sheffield.ac.uk}


\date{\today}

\dedicatory{}
\maketitle 

\begin{abstract}The Langlands Programme predicts that a weight $2$ newform $f$ over a number field $K$ with integer Hecke eigenvalues generally should have an {\em associated} elliptic curve $E_f$ over $K$. 
In our previous paper, we associated, building on works of Darmon \cite{darmon-integration} and Greenberg \cite{Gr}, a $p$-adic lattice $\Lambda$ to $f$, under certain hypothesis, and implicitly conjectured that $\Lambda$ is commensurable with the 
$p$-adic Tate lattice of $E_f$. In this paper, we present this conjecture in detail and discuss how it can be used to compute, directly from $f$, an explicit Weierstrass equation for the conjectural $E_f$.
 We develop algorithms to this end and implement them in order to carry out extensive systematic computations in which we compute Weierstrass equations of hundreds of elliptic curves, some with huge heights, over dozens of number 
fields. The data we obtain provide overwhelming amount of support for the conjecture and furthermore demonstrate that the conjecture provides an efficient tool to building databases of elliptic curves over number fields. 
\end{abstract}

\section{Introduction}
Perhaps one of the biggest achievements of number theory in the $20^{\rm th}$-century is the establishment of the correspondence between isogeny classes of elliptic curves over $\Q$ of a fixed conductor $N$ and weight $2$ newforms over $\Q$ of level $\Gamma_0(N)$ which have integer eigenvalues. This correspondence is believed to admit a suitable extension to general number fields and establishing this extension is one of the major goals of the Langlands Programme. In this paper, we are interested in one direction of this conjectural extension (see \cite{taylor-94} for 
the statement of the full conjectural correspondence).
\begin{conjecture} \label{conj: main} Let $K$ be a number field with ring of integers $\mathcal{O}_K$. Let $\fN$ be an ideal of $\mathcal{O}_K$.
For every weight $2$ newform $f$ over $K$ of level $\Gamma_0(\fN)$ with integer Hecke eigenvalues, there is either an 
elliptic curve $E_f/K$ of conductor $\fN$ such that 
$$ L(f,s)=L(E_f,s)$$ 
or a fake elliptic curve $A_f/K$ of conductor $\fN^2$ such that 
$$ L(f,s)^2=L(A_f,s).$$
\end{conjecture}

Recall that an abelian surface $A$ over $K$ is called a {\em fake elliptic curve} if 
$\End_K(A)\otimes \Q$ is isomorphic to a rational quaternion division algebra. Equivalently, it is a QM-abelian surface over $K$ whose $\overline{K}$-endomorphisms are all defined over $K$. The name, coined by Serre, comes from the fact that 
at any prime $\mathfrak{p}$ of good reduction, the reduction $A_\mathfrak{p}$ of $A$ is the square of an elliptic curve over the residue field 
of $\fp$. It is well known that if $A/K$ is a fake elliptic curve then $K$ is necessarily totally imaginary. 

For $K$ totally real, Conjecture \ref{conj: main} is known to hold when the Hecke eigenvalue system afforded by $f$ can be captured by a weight $2$ newform $f'$ on a Shimura curve via the Jacquet-Langlands correspondence (see \cite{Blasius}). Beyond totally real fields,  Conjecture \ref{conj: main} is completely open. However there are significant amounts of numerical data (see \cite{grunewald-helling-mennicke, cremona, whitley, bygott, lingham, scheutzow, gunnells_5, gunnells_23, jones, ellipticcurvesearch}) collected over the years that support its validity. In fact, when $K$ is a CM-field, one can {\em prove} whether a given elliptic curve is modular by a given newform. This employs the existence of 2-dimensional $p$-adic Galois representations associated to modular forms (see the recent breakthroughs \cite{scholze, taylor_etal}) and the Faltings--Serre method (see \cite{livne, dieulefait}). 

It is important to note that Conjecture \ref{conj: main} is not a constructive statement, it simply claims existence. And as such, it is not perfectly satisfying. It is desirable to have a {\em description} of 
$E_f$ whose existence is claimed by Conjecture \ref{conj: main}. In this regard, the situation is best over $K=\Q$. One can construct $E_f$ analytically as a torus $\C / \Lambda$ using the periods of $f$. Modular symbols give a method that allows for fast computations of these periods. The celebrated database of elliptic curves over $\Q$ of Cremona (\cite{lmfdb}) is based on an efficient implementation of the above strategy. Over totally real fields $K$, in cases where we can transfer the situation to a Shimura curve as mentioned above, one can describe the period lattice $\Lambda$ of $E_f$ using the periods of $f'$ this time. However, one runs into computational difficulties as one does not have the modular symbols method anymore due to the lack of cusps. Recently \cite{VW} and \cite{Nelson} have made progress in devising efficient algorithms to compute these periods.

Beyond the above situations, no description of $E_f$ is known. To date, the only case for which there has been a {\em conjectural} description is over totally real fields. In this case, there is a conjecture of Oda \cite{oda} which describes the period lattice of $E_f$ using the periods of the Hilbert modular form $f$. In \cite{Dembele, bober}, this conjecture was successfully utilized to compute the equation of $E_f$ in the case of $K=\Q(\sqrt{5})$. However, over general number fields $K$, the approach of trying to construct the period lattice $\Lambda$ of $E_f$ directly from the periods of $f$, as in the conjecture of Oda, runs into difficulty as the periods of $f$ do not suffice in general. For example, it is well known (see \cite{kurcanov, egm}) that the periods of a weight 2 Bianchi newform $f$ span a one-dimensional lattice in $\R$. Due to the lack of a (conjectural) description, all the numerical works that investigate the validity of Conjecture \ref{conj: main} that we alluded to above compile lists of elliptic curves over $K$ essentially by searching through the Weierstrass coefficients in a box. 

Our previous paper \cite{GMS} contains two constructions that lead to two conjectures; one describing the complex period lattice 
of $E_f$ over number fields with at least one real place, a second one describing the homothety class of the $p$-adic Tate lattice of $E_f$ when it admits one. In fact, the focus of~\cite{GMS} was in constructing certain local points on $E_f$ which were conjectured to be global (the so-called {\em Darmon points}), generalizing the seminal work of Darmon \cite{darmon-integration} to number fields of arbitrary signature in a cohomological manner that was pioneered by Greenberg \cite{Gr}. The conjectural description of the lattice (complex or $p$-adic) of $E_f$ was treated rather tangentially, its main role being that of providing well-definedness of the Darmon points. In particular, the experimental evidence supporting the lattice conjecture given in \cite{GMS} is essentially of an indirect nature, manly coming from the fact that the numerically computed Darmon points appeared to lie on a curve having the predicted lattice.

The focus of this paper is on the conjectural description of the $p$-adic lattice of $E_f$, in those cases where it admits a Tate uniformization. What we do is to present an explicit description of the construction which is better suited for numerical calculations. For number fields that have at most one complex place, we also provide efficient algorithms for computing the lattices in practice. In addition, we show how to use the information on the homothety class of the lattice to get $p$-adic approximations to the algebraic invariants of $E_f$, allowing for the recovery of the Weierstrass equation of $E_f$.

We have successfully implemented our approach and computed extensive tables that cover hundreds of elliptic curves over dozens of numbers fields. In particular, we have produced elliptic curves over number fields whose Weierstrass coefficients have huge heights. In certain situations this method can be more feasible than the ``searching methods'' that have been employed by many authors, as in~\cite{ellipticcurvesearch,bober}.

On the one hand, our data provide an overwhelming amount of data supporting the validity of the conjecture. On the other hand, they show the potential of the method for actually computing equations of elliptic curves attached to modular forms over number fields, which can be helpful in extending Cremona's tables to number fields other than $\Q$. Actually, the only reason why we focus on the $p$-adic conjecture is that the archimedean conjecture suffers from being ineffective from a computational point of view.

Let us give a sketch of our method and the contents of the paper. Let $K$ be a number field, which we assume of narrow class number $1$. Let $f$ be a weight 2 newform of over $K$ of level $\Gamma_0(\fN)$ with integer Hecke eigenvalues. Assume that there is a prime ideal $\fp$ such that $\fp || \fN$. In Section \ref{subsec: mod curves cohomology}, we start with transferring the problem into the realm of the cohomology of arithmetic groups. 
As it has both theoretical and computational advantages, we consider not only $\PGL_2$ but also its inner forms. We present a cohomological version of Conjecture \ref{conj: main} and then 
show that our assumption above on the level ideal $\fN$ of $f$ naturally takes fake elliptic curves out of the game.

In Section \ref{subsection: a conjectural construction of the lattice}, we expose the construction of the $\fp$-adic lattice. We first move things from arithmetic group setting to $S$-arithmetic group setting and single out a Hecke eigenclass $\Psi_f$ in the cohomology, with coefficients in the space of rigid analytic differential 1-forms on the $\fp$-adic upper half-plane $\cH_\fp$, of a certain $S$-arithmetic group that ``captures the arithmetic of $f$". Then we consider the $\fp$-adic lattice $\Lambda$ obtained by pairing, under a certain multiplicative integration pairing, $\Psi_f$ with certain homology classes with coefficients in degree zero divisors on $\cH_\fp$. The main conjecture of this paper then claims that $\Lambda$ is homothetic to the Tate lattice of an elliptic curve $E_f$ over $K$ which is modular by $f$. 

In Section \ref{sec: explicit methods and algorithms} we describe the methods for explicitly computing the $\fp$-adic lattice, in the case where $K$ has at most one complex place. What remains to be done is  extraction of the algebraic invariants of $E_f$ from $\Lambda$, which we discuss in Section \ref{sec: numerical computations and tables}. Let us say a few words about the implementation and the data collected. A more detailed discussion can be found in Sections \ref{sec: explicit methods and algorithms} and \ref{sec: numerical computations and tables}. All the data regarding the geometry of arithmetic groups, which were used as the input for our (co)homology programs, were obtained using programs of John Voight \cite{voight} (for arithmetic Fuchsian groups) and Aurel Page \cite{page} (for arithmetic Kleinian groups). In particular, we were only able to compute with number fields $K$ which were either totally real or almost totally real (that is, with a unique complex place). The $\fp$-adic integration pairing was computed using the tools developed by the first two authors in~\cite{gm}. These employ the method of {\em overconvergent cohomology} without which the necessary computations would not be feasible.

The homothety class of a lattice $\Lambda=q^\Z \subset K_\fp^\times$ is determined by its $\mathcal{L}$-invariant
$$ \mathcal{L}(\Lambda):=\dfrac{\log_\fp q }{\ord_\fp q}.$$
If an elliptic curve $E$ over $K$ admits a $\fp$-adic Tate uniformization $K_\fp^\times / q_E^\Z$, then its $\fp$-adic $\mathcal{L}$-invariant is defined as $\mathcal{L}(q_E^\Z)$. 
So the main conjecture of our paper describes the conjectural $E_f$ of Conjecture \ref{conj: main} through its $\fp$-adic $\mathcal{L}$-invariant. 

\section*{Acknowledgements}
The authors would like to thank Victor Rotger for suggesting the idea that lead to this project to the third named author in 2011. The authors would also like to thank John Cremona for his constant encouragement during this work. We would also like to express our gratitude to Aurel Page, who has helped us to tune his code, a crucial input to our algorithms. Finally we thank John Voight for his detailed comments on an earlier version of this paper which we found very stimulating.

\section{Modular elliptic curves and $\fp$-adic lattices}\label{sec: modular curves and lattices}
We begin this section by restating Conjecture \ref{conj: main} in terms of certain classes in the cohomology of quaternion orders. Then we construct $p$-adic lattices from these cohomology classes which, conjecturally, correspond to the associated modular elliptic curves.

\subsection{Modular elliptic curves via cohomology}\label{subsec: mod curves cohomology} The followings are well known, we invite the reader to consult \cite{harder-87}, \cite[\S3]{hida94} or \cite[\S2]{GMS} for definitions or details that are missing. Let $K$ be a number field which we assume to be of narrow class number $1$. Denote by $(r,s)$ its signature; i.e., $K$ has $r$ real places and $s$ complex places. Let $B$ be a quaternion algebra over $K$ of discriminant $\fD$ and which splits at $n$ of the real places. Thus there is an embedding 
\begin{align}\label{eq: embedding of B}B^\times/K^\times\hookrightarrow \PGL_2(\R)^n\times\PGL_2(\C)^s,\end{align}
given by the choice of splitting isomorphisms at the archimedean places.

If $\fm$ is an integral ideal of $K$ which is coprime to $\fD$, we denote by $R_0^\fD(\fm)$ an Eichler order of level $\fm$. We set $\tilde\Gamma_0^\fD(\fm)=R_0^\fD(\fm)^\times/\cO_K^\times$ and assume that it is torsion-free. The group $\PGL_2(\R)$ acts on the upper half-pane $\cH=\R \times\R_{>0}$ by fractional linear transformations (for negative determinant matrices, we first apply complex conjugation). Similarly, $\PGL_2(\C)$ acts on the hyperbolic $3$-space $\mathbb{H} =\C\times\R_{>0}$. Therefore, $\tilde\Gamma_0^\fD(\fm)$ acts via \eqref{eq: embedding of B} on $\cH^n\times \mathbb{H}^s$ and the quotient
\begin{align}
  Y_0^\fD(\fm) = \tilde\Gamma_0^\fD(\fm)\backslash \cH^n\times \mathbb{H}^s
\end{align}
is a Riemannian manifold of real dimension $2n+3s$, which is non-compact if and only if the ambient quaternion algebra $B$ is the $2\times 2$ matrix algebra over $K$.

For any abelian group $A$ with a $\tilde\Gamma_0^\fD(\fm)$-action, the Betti cohomology groups $H^{n+s}( Y_0^\fD(\fm),A)$ are finitely generated abelian groups. As 
$\cH^n\times \mathbb{H}^s$ is contractible, we have
\begin{align*}
 H^{i}( Y_0^\fD(\fm),A) \simeq   H^{i}(\tilde\Gamma_0^\fD(\fm),A),
\end{align*}
where the cohomology on the right is group cohomology. We shall often interchange the two sides without alerting the reader. 

These cohomology groups are equipped with the action of the Hecke operators, a collection of endomorphism $\{T_\fl\}$ indexed by the primes $\fl\nmid \fD$. A cohomology class $f\in H^{n+s}( Y_0^\fD(\fm),\C)$ is said to be a \emph{Hecke eigenclass} if it is an eigenvector for all the Hecke operators and is said to be \emph{rational} if all its eigenvalues are integers. That is,
\begin{align}\label{eq: eigenclass}
  T_\fl f = a_\fl(f) f \text{ with $a_\fl(f)\in \Z$, for all $\fl\nmid \fD$.}
\end{align}
We say that such an $f$ is \emph{Eisenstein} if $a_\fl(f)= |\fl| + 1$ for all $\fl\nmid \fD$, where $|\fl|$ stands for the norm of the ideal $\fl$. Two Hecke eigenclasses $f$ and $f'$, possibly of different levels $\fm$ and $\fm'$, are said to be \emph{equivalent} if $a_\fl(f) = a_\fl(f')$ for all $\fl\nmid \fm\fm'$. We say that $f$ is \emph{new} if it is {\em not} Eisenstein and {\em not} equivalent to any Hecke eigenclass of level a strict divisor of $\fm$. 

The generalized Eichler-Shimura Isomorphism, as established by Harder, and the Jacquet-Langlands Correspondence tell us, roughly speaking, that the cohomology groups 
$H^{n+s}( Y_0^\fD(\fm),\C)$ correspond to weight $2$ modular forms of level $\Gamma_0(\fm)$. 
The following conjecture is a cohomological version of Conjecture \ref{conj: main}, in which we consider not only arithmetic groups for $\PGL_2$ over $K$ but also its inner forms. 

\begin{conjecture}\label{conj: cohomological C1}
  Let $f\in H^{n+s}( Y_0^\fD(\fm),\C)$ be a new rational Hecke eigenclass. If $K$ has some real place, then there exists an elliptic curve $E_f/K$, of conductor $\fD\fm$, such that
  \begin{align}\label{eq: property defining E_f}
    \# E_f(\cO_K/\fl) = 1 + |\fl| - a_\fl(f) \text{ for all } \fl\nmid \fD.
  \end{align}
If $K$ is totally complex, then there exists either and elliptic curve $E_f$ of conductor $\fD\fm$ satisfying \eqref{eq: property defining E_f} or a fake elliptic curve $A_f/K$, of conductor $(\fD\fm)^2$, such that
  \begin{align}\label{eq: property defining A_f}
    \# A_f(\cO_K/\fl) = (1 + |\fl| - a_\fl(f))^2 \text{ for all } \fl\nmid \fD.
  \end{align}
\end{conjecture}
\begin{remark}
Observe that condition \eqref{eq: property defining E_f} does not uniquely characterize $E_f$, but only its $K$-isogeny class. We will abuse the terminology and refer to any curve satisfying \eqref{eq: property defining E_f} as the curve $E_f$ associated to $f$. A similar remark holds for $A_f$.
\end{remark}

\begin{remark}\label{rk: restate as Gamma_0}
It is sometimes convenient to work with the group $\Gamma_0^\fD(\fm)=R_0^\fD(\fm)^\times_1/\{\pm 1\}$ (here $R_0^\fD(\fm)^\times_1$ denotes the group of elements of reduced norm $1$ in $R_0^\fD(\fm)^\times$). In this case, one needs to take into account the involutions coming from units of $K$. More precisely, denote by $U_+'$ the units in $\cO_K^\times$ which are positive at the real places that ramify in $B$. Then any any representative $u\in U'_+/(\cO_K^\times)^2$ gives rise to an involution $T_u$ on $H^{n+s}( \Gamma_0^\fD(\fm),\C)$ 
(see \cite[\S2]{GMS}) and it follows that $H^{n+s}( \tilde{\Gamma}_0(\fm),\C)$ can be identified with the subspace of $H^{n+s}(\Gamma_0^\fD(\fm),\C)$ that is fixed under 
these involutions. Hecke operators $\{T_\fl\}_{\fl\nmid \fD}$ on $H^{n+s}( \Gamma_0^\fD(\fm),\C)$ are defined in the usual way and a Hecke eigenclass 
in $H^{n+s}(\tilde\Gamma_0^\fD(\fm),\C)$ can be thought of as a Hecke eigenclass $H^{n+s}(\Gamma_0^\fD(\fm),\C)$ that is fixed by the involutions 
$T_u$ with $u\in U_+'/(\cO_K^\times)^2$. 
\end{remark}


From Section \ref{subsection: a conjectural construction of the lattice} on we will consider levels that have valuation $1$ at some prime. The following proposition  and corollary rule out the possibility of having a fake elliptic curve in that setting.
\begin{proposition}\label{prop: about fake elliptic curves}
 Let $A$ be a fake elliptic curve over $K$ of conductor $\mathfrak{L}$, and let $\fp$ be a prime dividing $ \mathfrak{L}$. Then $v_\fp(\mathfrak{L})\geq 4$.
\end{proposition}
\begin{proof}
  Let $A'$ denote the connected component of the special fiber at $\fp$ of the N\'eron model of $A$. By the Chevalley theorem on algebraic groups there is an exact sequence
  \begin{align*}
    0 \lra T\times U \lra A' \lra  B \lra 0,
  \end{align*}
with $B$ an abelian variety, $T$ a torus, and $U$ a unipotent group. Denote by $t$ the dimension of $T$, and by $u$ the dimension of $U$. The valuation of $\mathfrak{L}$ at $\fp$ is given by 
\begin{align*}
v_\fp(\mathfrak{L}) =  t +2u + d_\fp,
\end{align*}
where $d_\fp$ is the Swan conductor. It is well known that $A$ has potentially good reduction (see, e.g., \cite[Theorem $3$]{RibetEnds}), so that $t=0$. Since $A$ has bad reduction at $\fp$, we have that necessarily $u>0$. Therefore, in order to finish the proof, we need to rule out the case $u=1$.

If $u=1$ then $B$ is an elliptic curve. Any endomorphism of $A$ gives rise to an endomorphism of its Neron model, so that $\End(A)$ acts on $A'$. By functoriality we see that $\text{End}(A)$ also acts on $B$ (this follows, for instance, from the fact that $B$ is the Albanese variety of $A'$). Thus $B$ must be a supersingular elliptic curve and, moreover, $\End(B)\otimes \Q\simeq \End(A)\otimes \Q$. But this is impossible: by a theorem of Tate (\cite[Theorem 2]{tate-endomorphisms}), the endomorphism algebra of a supersingular elliptic curve is ramified at $\infty$; on the other hand, that of a fake elliptic curve is well known to be split at $\infty$ (this follows from results of Shimura \cite{shimura-analytic-familes}).
\end{proof}
\begin{corollary}
Let $f\in H^{n+s}(\Gamma_0^\fD(\fp\fm),\C)$ be a rational Hecke eigenclass, where $\fp$ is a prime that does not divide $\fD\fm$. Then the abelian variety associated to $f$ in Conjecture \ref{conj: cohomological C1} is an elliptic curve, rather than a fake elliptic curve.
\end{corollary}
\begin{proof}
  If it was a fake elliptic curve its conductor would be $\mathfrak{L}=\fp^2\fm^2\fD^2$, with $\fp\nmid \fm\fD$ and therefore $v_\fp(\mathfrak{L})=2$. This would contradict Proposition \ref{prop: about fake elliptic curves}.
\end{proof}

\subsection{Construction of the $\fp$-adic lattice}\label{subsection: a conjectural construction of the lattice} Let $\fp$ be a prime of $K$ and put $\C_\fp= \widehat{\overline K}_\fp$. Recall Tate's uniformization: If $E$ is an elliptic curve over $K$ whose conductor is exactly divisible by $\fp$, there exists a lattice $\Lambda\subset \C_\fp^\times$ such that $E(\C_\fp)\simeq \C_\fp^\times/\Lambda$.

Let $\fn$ be an ideal coprime to $\fp\fD$, and let $f\in H^{n+s}(\Gamma_0^\fD(\fp\fn),\C)$ be a new rational Hecke eigenclass. The goal of this section is to construct a lattice $\Lambda_f\subset \C_\fp^\times$ which we conjecture is homothetic to the Tate lattice of some elliptic curve $E_f$ over $K$ that is modular by $f$. This will be done in subsection \ref{subsec: construction of the lattice} below. Before that, we briefly recall some of the tools that will be used and we fix some notation.

\subsubsection{Arithmetic and $S$-arithmetic groups} Let $R_0^\fD(\fn)$ and $R_0^\fD(\fp\fn)$ denote Eichler orders in $B$ of the indicated levels, chosen in such a way that $R_0^\fD(\fp\fn)\subset R_0^\fD(\fn)$. We set $\Gamma_0^\fD(\fp\fn)=R_0^\fD(\fp\fn)^\times_1/\{\pm 1\}$ and $\Gamma_0^\fD(\fn)=R_0^\fD(\fn)^\times_1/\{\pm 1\}$.

For a set of primes $S$ of $\cO_K$ we let $\cO_{K,S}$ denote the $S$-integers of $K$, that is the set of $x\in K$ such that $v_\fq(x)\geq 0$ for all primes $\fq\not\in S$. We put $R = R_0^\fD(\fp\fn)\otimes_{\cO_K}\cO_{K,\{\fp\}}$ and $\Gamma = R^\times_1$. Observe that $\Gamma$ is an $S$-arithmetic group that contains the arithmetic groups $\Gamma_0^\fD(\fp\fn)$ and $\Gamma_0^\fD(\fn)$.

\subsubsection{(Co)homology groups and Hecke operators}\label{subsec: cohomology groups and hecke operators} If $V$ is a $R_0^\fD(\fp\fn)^\times-$module the groups $H^{i}(\Gamma_0^\fD(\fp\fn),V)$ and $H_i(\Gamma_0^\fD(\fp\fn),V)$ are endowed with the action of Hecke operators $\{T_\fl\}$ for ${\fl\nmid \fD}$. Following the usual notational conventions, we set $U_\fl = T_\fl$ for $\fl\mid \fp\fn$. As noted in Remark \ref{rk: restate as Gamma_0}, one also has involutions $T_u$ associated to units $u\in U'_+/(\cO_K^\times)^2$. Moreover, there are Atkin-Lenher involutions $W_\fl$ at the primes $\fl\mid \fp\fn$. For instance, if $\pi$ is a generator of $\fp$ which is positive at the real places of $K$, then $W_\fp$ is induced by an element $\omega_\pi\in R_0^\fD(\fp\fn)^\times$ of reduced norm $\pi$ and which normalizes $\Gamma_0^\fD(\fp\fn)$. 

If $G$ denotes either $\Gamma_0^\fD(\fn)$ or $\Gamma$, then there are analogous Hecke and Atkin--Lehner operators acting on $H^i(G,V)$ and $H_i(G,V)$.


\subsubsection{Bruhat--Tits tree, harmonic cocycles,  and measures}
Let $\cT$ denote the Bruhat--Tits tree of $\PGL_2(K_\fp)$. Its set of vertices $\cV$ is identified with the set of homothety classes of $\cO_{K_\fp}$-lattices in $K_\fp^2$. Its set of directed edges $\cE$ consists on ordered pairs $(w_1,w_2)\in \cV\times\cV$ such that each $w_i$ can be represented by a lattice $\Lambda_i$ with $\fp \Lambda_1\subsetneq \Lambda_2\subsetneq \Lambda_1$. The natural action of $\PGL_2(K_\fp)$ on the lattices induces an action on $\cT$.

For $e=(w_1,w_2)\in \cE$ we let $s(e)=w_1$ denote its source, $t(e)=w_2$ its target, and $\bar e= (w_2,w_1)$ its opposite.  Let $v_0\in\cV$ be the vertex corresponding to $\cO_{K_\fp}\times \cO_{K_\fp}$ and $v_1$ that corresponding to $\cO_{K_\fp}\times\fp\cO_{K_\fp}$, and let $e_0\in\cE$ be the directed edge $(v_0,v_1)$. We denote by $\cV_0$ the set of even vertices (i.e., those at an even distance of $v_0$), and by $\cV_1$ the set of odd vertices. Similarly, $\cE_0$ stands for the set of even edges (those $e$ such that $s(e)$ is even) and $\cE_1$ for the odd edges.

We can, and do, fix a splitting isomorphism
\begin{align}
  \iota_\fp \colon B\otimes_K K_\fp \lra M_2(K_\fp)
\end{align}
such that $\iota_p(R_0^\fD(\fn))\simeq M_2(\cO_{K_\fp})$ and $\iota_\fp(R_0^\fD(\fp\fn))\simeq \{\smtx a b c d \in\M_2(\cO_{K_\fp})\colon c\in\fp\}$. In this way we identify $\Gamma_0^\fD(\fp\fn)$, $\Gamma_0^\fD(\fn)$, and $\Gamma$ with their images in $\PGL_2(K_\fp)$ under $\iota_\fp$ , so that they acquire an action on $\cT$. It turns out that $\Gamma$ acts transitively on $\cE_0$, and this gives rise to one-to-one correspondences
\begin{align}\label{eq: correspondences}
  \Gamma/\Gamma_0^\fD(\fn)\leftrightarrow \cV_0  \  \text{ and } \ \Gamma/\Gamma_0^\fD(\fp\fn)\leftrightarrow \cE_0,
\end{align}
induced by $g\mapsto gv_0$ and $g\mapsto g e_0$, respectively. Similarly, if we set $\widehat\Gamma_0^\fD(\fn) = \omega_\pi \Gamma_0^\fD(\fn) \omega_\pi^{-1}$ then the maps $g\mapsto gv_1$ and $g\mapsto ge_1$ induce bijections
\begin{align}\label{eq: correspondences 2}
  \Gamma/\widehat\Gamma_0^\fD(\fn)\leftrightarrow \cV_1  \  \text{ and } \ \Gamma/\Gamma_0^\fD(\fp\fn)\leftrightarrow \cE_1.
\end{align}

As a consequence of \eqref{eq: correspondences}, for any abelian group $A$ (with trivial $\Gamma_0^\fD(\fn)$-action) we have isomorphisms 
\begin{align}\label{eq: induced modules}
\Ind_{\Gamma_0^\fD(\fp\fn)}^\Gamma A\simeq \cF(\cE_0,A)\ \text{ and }\  \Ind_{\Gamma_0^\fD(\fn)}^\Gamma(A)\simeq \cF(\cV_0,A),
\end{align}
where $\Ind$ stands for the induced module and $\cF(X,Y)$ for the set of functions from $X$ to $Y$. Similar isomorphisms are deduced from \eqref{eq: correspondences 2}.

Let $\cF_0(\cE,A)$ be the set of functions $\mu\colon \cE\ra A$ such that $\mu(e) + \mu(\bar e)=0$. There are two degeneracy maps $\varphi_s,\varphi_t\colon \cF(\cE,A)\ra \cF(\cV,A)$ given by 
\begin{align*}\varphi_s(\mu)(v)= \sum_{s(e)=v} \mu(e)\ \text{ and } \ \varphi_t(\mu)(v)= \sum_{t(e)=v} \mu(e).
\end{align*}
The map $\varphi_s$ sends $\cF_0(\cE,A)$  exhaustively onto $\cF(\cV,A)$. The group of \emph{$A$-valued harmonic cocycles}, denoted $\text{HC}(A)$, is defined to be the kernel; that is, it is defined by the exact sequence
\begin{align*}
  0\lra \text{HC}(A)\lra \cF_0(\cE,A)\lra \cF(\cV,A)\lra 0.
\end{align*}

Denote by $\text{Meas}_0(\P^1(K_\fp),A)$ the set of $A$-valued measures on $\P^1(K_\fp)$ with total measure $0$. If $\cB$ is the set of compact-open balls in $\P^1(K_\fp)$ then, by definition, $\mu\in \text{Meas}_0(\P^1(K_\fp),A)$ is a function $\mu\colon \cB\ra A$ such that $f(\P^1(K_\fp))=0$ and, for any $B\in\cB$, the following compatibility condition holds:
\begin{align}\label{eq: compatibility of the measures}
 \text{ if $B=\bigsqcup B_i$ is a finite decomposition with $B_i\in\cB$, then $\mu(B)=\sum \mu(B_i)$}.
\end{align}
One consequence is that, in particular,  $\mu (\P^1(K_\fp)\setminus B)= -\mu(B)$. 

An example of compact-open ball in $\P^1(K_\fp)$ is $\cO_{K_\fp}$, the ring of integers of $K_\fp$. Given $B\in \cB$, either $B$ or $\P^1(K_\fp)\setminus B$ can be expressed as $\gamma\cO_{K_\fp}$ for some $\gamma\in \Gamma$. Moreover, the stabilizer of $\cO_{K_\fp}$ in $\Gamma$ is $\Gamma_0^\fD(\fp\fn)$. This facts, together with \eqref{eq: correspondences}, imply that any $\mu\in \text{Meas}_0(\P^1(K_\fp),\Z)$ can be identified with a function $\mu\colon \cE_0 \ra A$ satisfying the compatibility coming from \eqref{eq: compatibility of the measures}. This compatibility condition turns out to be that of being a harmonic cocycle. Therefore, we have an isomorphism $\text{Meas}_0(\P^1(K_\fp),A)\simeq \text{HC}(A)$, which we will use to identify measures and harmonic cocycles from now on.

\subsubsection{Multiplicative integration pairing}
 Let $\mathcal{C}(\P^1(K_\fp),\Z)$ denote the $\Z-$valued continuous functions on $\P^1(K_\fp)$. For $\mu\in \text{Meas}_0(\P^1(K_\fp),\Z)$  and $f\in \mathcal{C}(\P^1(K_\fp),\Z)^\times$ one defines the \emph{multiplicative integral} of $f$ with respect to $\mu$ to be 
\begin{align}\label{eq: mult int}
  \Xint\times_{\P^1(K_\fp)} f d\mu = \lim_{||\cU||\ra 0}\prod_{U\in \cU} f(t_U)^{\mu(U)},
\end{align}
where $\cU$ runs over coverings of $\P^1(K_\fp)$ by compact open balls whose diameter approaches to $0$, and $t_U$ is any sample point in $U$.

Let $\cH_\fp=\C_\fp\setminus K_\fp$ denote the $\fp$-adic upper half plane. Using the above multiplicative integrals one defines the following pairing:
\begin{align*}
\begin{array}{ccc}
  \text{Meas}(\P^1(K_\fp),\Z)  \times   \Div^0(\cH_\fp) & \lra & \C_\fp^\times\\
(\mu  ,  \tau_1-\tau_2) & \longmapsto &\displaystyle \Xint\times_{\P^1(K_\fp)} \left( \frac{t-\tau_2}{t-\tau_1} \right)d\mu.
\end{array}
\end{align*}
This induces, by cap product, a \emph{multiplicative integration pairing} between $\Gamma$-(co)homology groups:
\begin{align}
  \Xint\times \langle\ ,\  \rangle \colon H^i(\Gamma,\text{HC}(\Z))\times H_i(\Gamma,\Div^0\cH_\fp)\lra \C_\fp^\times.
\end{align}

Denote by $\Omega^1_{\cH_\fp}(\Z)$ the $\Z$-module of rigid-analytic $1$-forms on $\cH_\fp$ for which all of their residues are in $\Z$. 
It is well known that by considering residues of harmonic cocycles on carefully chosen annuli in $\cH_\fp$, one can exhibit an isomorphism between $HC(\Z)$ and $\Omega^1_{\cH_\fp}(\Z)$ 
(see \cite[\S4.2]{GMS}). Thus we may as well consider the above integration pairing using cohomology classes with coefficients in $\Omega^1_{\cH_\fp}(\Z)$, as done in \cite{GMS} and alluded to in the Introduction. 

\subsubsection{Construction of the lattice}\label{subsec: construction of the lattice}
Recall the rational Hecke eigenclass $f\in H^{n+s}(\Gamma_0^\fD(\fp\fn),\Z)$. In this section we construct a lattice $\Lambda_f\subset\C_\fp^\times$. More precisely, $\Lambda_f$ will be the lattice generated by a quantity $q_f\in\C_\fp^\times$ that will be defined as
\begin{align*}
  q_f = \Xint\times\langle \omega_f,\Delta_f\rangle, 
\end{align*}
for certain cohomology class $\omega_f\in H^{n+s}(\Gamma,\text{HC}(\Z))$ and homology class $\Delta_f\in H_{n+s}(\Gamma,\Div^0\cH_\fp)$. Next, we give the definition of $\omega_f$ and $\Delta_f$.

By Shapiro's Lemma and \eqref{eq: induced modules} we have the following isomorphisms:
\begin{align}\label{eq: shapiro}
  H^{n+s}(\Gamma_0^\fD(\fp\fn),\Z) \simeq H^{n+s}(\Gamma,\Ind_{\Gamma_0^\fD(\fp\fn)}^\Gamma \Z)\simeq H^{n+s}(\Gamma,\cF(\cE_0,\Z)).
\end{align}
By definition of harmonic cocycles we have an inclusion $\text{HC}(\Z)\subset \cF_0(\cE,\Z)\simeq \cF(\cE_0,\Z)$. Therefore, there is a natural map
\begin{align*}
  \rho\colon H^{n+s}(\Gamma,\text{HC}(\Z))\lra H^{n+s}(\Gamma,\cF(\cE_0,\Z)).
\end{align*}
It turns out that the class corresponding to $f$ under the identifications \eqref{eq: shapiro} lies in the image of $\rho$, and we define $\omega_f\in H^{n+s}(\Gamma,\text{HC}(\Z))$ to be an element such that $\rho(\omega_f)= f$.

By duality between homology and cohomology groups, the Hecke eigenclass $f\in H^{n+s}(\Gamma_0^\fD(\fp\fn),\Z)$ gives rise to $\hat f\in H_{n+s}(\Gamma_0^\fD(\fp\fn),\Z)$ characterized, up to scaling, by the fact that it has the same eigenvalues as $f$ for all the Hecke operators.

Since $\Gamma$ is isomorphic to the amalgamated product $ \Gamma^\fD_0(\fn)\star_{\Gamma_0^\fD(\fp\fn)}\Gamma_0^\fD(\fn)$, the corresponding Mayer--Vietoris sequence gives, at degree $n+s$:
\begin{align}\label{eq: exact sequence in homology} 
\cdots \lra  H_{n+s+1}(\Gamma,\Z)\stackrel{d}{\lra}  H_{n+s}(\Gamma_0^\fD(\fp\fn),\Z) \stackrel{\partial_*}{\lra} H_{n+s}(\Gamma_0^\fD(\fn),\Z)^2\lra \cdots .
\end{align}

The fact that $\hat f$ is new at $\fp$ is equivalent to $\partial_*(\hat f)=0$. Therefore, there exists a homology class $c_f\in H_{n+s+1}(\Gamma,\Z)$ such that $d(c_f) = \hat f$.

Now consider the exact sequence defining $\Div^0(\cH_\fp)$:
\begin{align*}
  0\lra \Z \lra \Div \cH_\fp \stackrel{\deg}{\lra} \Div^0 \cH_\fp \lra 0.
\end{align*}
The homology exact sequence gives a connecting homomorphism
\begin{align}\label{eq: connecting homomorphism in homology}
\delta\colon  H_{n+s+1}(\Gamma,\Z)\lra H_{n+s}(\Gamma,\Div^0\cH_\fp),
\end{align}
and we define $\Delta_f = \delta(c_f)$.

Finally, we define the period $q_f$ by
\begin{align*}
  q_f = \Xint\times \langle \omega_f , \Delta_f \rangle \in \C_\fp^\times,
\end{align*}
and the $\fp$-adic lattice $\Lambda_f = q_f ^\Z\subset \C_\fp^\times$.
\begin{conjecture}\label{conj: our main}
  The lattice $\Lambda_f$ is commensurable with the Tate lattice of an elliptic curve $E_f$ over $K$ which is modular by $f$.
\end{conjecture}
\begin{remark}
The above conjecture is known for $K=\Q$: for $B=\M_2(\Q)$ it was proven by Darmon \cite[Theorem 1]{darmon-integration}, who showed that in fact is equivalent to the Mazur--Tate--Teitelbaum conjecture, now a theorem of Greenberg--Stevens \cite{greenberg-stevens}; for $B$ a quaternion division algebra over $\Q$ it was proven by Dasgupta--Greenberg \cite{greenberg-dasgupta} and, independently, by Longo--Rotger--Vigni \cite{LRV}. Conjecture \ref{conj: our main} was stated for totally real $K$ in \cite{Gr} and for $K$ of arbitrary signature in \cite{GMS}. To the best of our knowledge, in these cases it remains open.
\end{remark}

\section{Explicit methods and algorithms}\label{sec: explicit methods and algorithms}
 In this section we describe explicit algorithms for computing $\Lambda_f$, in the particular case that $n+s=1$. Observe that $n+s$ is the degree of the (co)homology groups involved in the construction of $\Lambda_f$, and this is precisely the reason why we impose this restriction: we want to work with (co)homology groups of degree $1$, because they are easier to handle computationally. 

Recall that a number field is said to be almost totally real (ATR for short) if it has one complex place. That is, if it is of signature $(r,1)$ for some $r\geq 0$. The condition $n+s = 1$ implies that $K$ must be either totally real or almost totally real, which we assume from now on. 

\begin{remark}
We believe that it should be possible to extend the algorithms of this section to (co)homology groups of degrees $>1$, and that this would be interesting because it would allow to do computations in fields $K$ of arbitrary signature. However, we have not made any serious attempt in this direction.
\end{remark}

The input for the algorithms of this section is the following: a quaternion algebra $B/K$ of discriminant $\fD$ which is split at one archimedean place, an ideal $\fn$ coprime to $\fD$, and a prime $\fp$ such that $\fp\nmid \fn\fD$. The aim is to compute
\begin{enumerate}
\item the rational Hecke eigenclasses $f\in H^1(\Gamma_0^\fD(\fp\fn),\Z)$ or, equivalently, the rational Hecke eigenclasses $\hat f\in H_1(\Gamma_0^\fD(\fp\fn),\Z)$.
\end{enumerate}
 Then, for each rational Hecke eigenclass (if any) we compute 
\begin{enumerate}
\setcounter{enumi}{1}
\item the homology class $\Delta_f\in H_1(\Gamma,\Div^0\cH_\fp)$, 
\item the cohomology class $\omega_f\in H^1(\Gamma,\text{HC}(\Z))$, and 
\item the period $q_f = \Xint\times\langle \omega_f,\Delta_f \rangle\in \C_\fp^\times$. 
\end{enumerate}
We will take for granted the algorithms for working with quaternion algebras and their orders (cf., e.g., \cite{JVidentifying}), for instance those implemented in Magma \cite{magma}. Key to the methods that we present in this section are also the algorithms for computing arithmetic groups of the form $\Gamma_0^\fD(\fm)$. For quaternion algebras over totally real fields they are due to John Voight \cite{voight}, and over almost totally real fields to Aurel Page \cite{page}. In particular, we assume that there are algorithms for computing a presentation of $\Gamma_0^\fD(\fm)$ in terms of generators and relations and to solve the word problem, that is, any $g\in\Gamma_0^\fD(\fm)$ can be effectively expressed in terms of the generators.

Let us also fix some notation and conventions regarding homology and cohomology groups. Let $G$ denote a group and $A$ an abelian $G$-module. We will work with the so called bar resolution, in which the group of $i$-chains is taken to be $\Z[G]\otimes_\Z\stackrel{i)}{\cdots}\otimes_\Z\Z[G]\otimes_\Z A$. The boundary maps, which we only need in degrees $1$ and $2$, are given by
\begin{align*}
  \partial_1(g\otimes a) = ga -a, \  \text{ and }\ \partial_2(g_1\otimes g_2\otimes a) = g_2\otimes g_1^{-1} a -g_1g_2\otimes a +g_1\otimes a.
\end{align*}
For cohomology, the (inhomogeneous) $i$-cochains are the maps from $G^i$ with values in $A$, and the coboundaries in degrees $0$ and $1$ are
\begin{align*}
  \partial^0(a)(g) = g^{-1}a-a, \  \text{ and }\ \partial^1(c)(g_1,g_2) = g_1c(g_2)-c(g_1g_2)+c(g_1).
\end{align*}
\subsection{The rational Hecke eigenclass}\label{subsec: the rational eigenclass} Finding rational Hecke eigenclasses amounts to compute matrices of Hecke operators acting on $H^1(\Gamma_0^\fD(\fp\fn),\Q)$ or $H_1(\Gamma_0^\fD(\fp\fn),\Q)$. For totally real number fields, one can use the algorithms of \cite{greenberg-voight}, which in fact are valid more generally for cohomology groups of degree $>1$. In this section we use the explicit presentations and solutions to the word problem provided by \cite{page} to treat also the case of ATR fields, but only in (co)homological degree $1$. We present the methods just for homology, although everything can be easily adapted to cohomology as well.

The main idea is that homology in degree $1$ is the same as the abelianized of the group. Indeed, for any group $G$ there is a canonical isomorphism $\phi\colon G_\text{ab}\simeq H_1(G,\Z)$. If we identify the abelianized $G_\text{ab}$ with $G/[G,G]$ (here $[G,G]$ is the derived subgroup), and $H_1(G,\Z)$ with $\Z[G]/\partial_2(\Z[G]\otimes\Z[G])$, then $\phi$ is induced by the map (which, by abuse of notation, we also call $\phi$)
\begin{align}\label{eq: phi}
  \begin{array}{cccc}
\phi \colon &  G & \lra & \Z[G]\\
 & g & \longmapsto & g.
\end{array}
\end{align}
Using the algorithms of \cite{voight} and \cite{page} we can compute a presentation for $\Gamma_0^\fD(\fp\fn)$ of the form
\begin{align*}
\Gamma_0^\fD(\fp\fn) = \langle u_1,\dots,u_b \mid  r_1,\dots ,r_c \rangle, 
\end{align*}
where the $u_i$'s are generators and the $r_j$'s relations. From this, it is easy to compute generators $\{v_1,\dots, v_e\}$ for $\Gamma_0^\fD(\fp\fn)_{\text{ab}}$. Suppose that $v_1,\dots,v_d$ are of infinite order and the rest are torsion. That is to say, $\Gamma_0^\fD(\fp\fn)_\text{ab}\simeq \Z^d\oplus \text{Torsion}$. The torsion part is not important, as we are actually interested in the Hecke action on $H_1(\Gamma_0^\fD(\fp\fn),\Q)$. Therefore, for a prime $\fl\nmid \fp\fn\fD$ the Hecke operator $T_\fl$ will be described by a matrix $M(T_\fl)\in \M_d(\Z)$, which we next explain how to compute.

Let $\pi_\fl\in R_0^\fD(\fp\fn)$ be an element whose reduced norm generates $\fl$ and is positive under the real embeddings of $K$. Then there is a decomposition 
\begin{align*}
  \Gamma_0^\fD(\fp\fn)\pi_\fl\Gamma_0^\fD(\fp\fn) = \bigsqcup_{i=0}^{|\fl|} g_i\Gamma_0^\fD(\fp\fn).
\end{align*}
Since we know that there are $|\fl|+1$ cosets, the $g_i$'s are easy to find in practice. Indeed, all of them are of the form $g\pi_\fl$ with $g\in\Gamma_0^\fD(\fp\fn)$. One can run over different $g$'s in $\Gamma_0^\fD(\fp\fn)$ and check for equivalency modulo $\Gamma_0^\fD(\fp\fn)$ on the right, until $|\fl|+1$ inequivalent cosets are found.

Now, for each $i=0,\dots,|\fl|$ let $t_i\colon \Gamma_0^\fD(\fp\fn)\ra \Gamma_0^\fD(\fp\fn)$ be the map defined by means of the equation
\begin{align*}
  h^{-1}g_i = g_{h(i)}t_i(h)^{-1},
\end{align*}
for some index $h(i)\in \{0,\dots,|\fl|\}$. Suppose that $A$ is a $R_0^\fD(\fp\fn)^\times$-module and let $c=\sum_h h\otimes a_h\in \Z[\Gamma_0^\fD(\fp\fn)]\otimes A$ be a cycle. We denote by $[c]$ the class of $c$ in $H_1(\Gamma_0^\fD(\fp\fn),A)$.  Then a cycle representing $T_\fl([c])$ is given by the following formula (cf. \cite[\S 1]{ash-stevens}):
\begin{align}\label{eq: T_l}
  T_\fl([c]) = \sum_{i=0}^{|\fl|}\sum_{h} t_i(h)\otimes g_i^{-1} a_h.
\end{align}

Each generator $v_i$ gives rise to a cycle $[v_i]\in Z_1(\Gamma_0^\fD(\fp\fn),\Z)$. Then formula \eqref{eq: T_l} gives explicitly $T_\fl([v_i])$, regarded as an element of $\Z[G]$. It corresponds, via $\phi$, to an element of $\Gamma_0^\fD(\fp\fn)_\text{ab}$ which, using an algorithmic solution to the word problem of \cite{voight} and \cite{page}, we can express as $\sum_{j=1}^e a_{ji}v_j$ for some integers $a_{ji}$. Since we are only interested in the non-torsion generators, we just disregard the part corresponding to torsion and then the $i$-th column of $M(T_\fl)$ is given by $(a_{1i},\dots,a_{di})^t$.

Similarly, for any $u\in U_+'$, let $\omega_u\in R_0^\fD(\fp\fn)^\times$ be an element of reduced norm $u$. The involution $T_u$ is given by the formula
\begin{align*}
 T_u([c])=\sum_h \omega_u^{-1} h \omega_u\otimes \omega_u^{-1}a_g,
\end{align*}
and we can compute its matrix $M(T_u)\in\M_d(\Z)$ by the same procedure as with the Hecke operators at finite primes $T_\fl$.

Now, in order to determine the rational Hecke eigenclasses one decomposes the free part of $\Gamma_0^\fD(\fp\fn)_\text{ab}$ into simultaneous eigenspaces with respect to the action of $M(T_u)$, for all $u\in U_K'/(\cO_K^\times)^2$, and the matrices $M(T_\fl)$, for several $\fl$'s until all the eigenspaces are irreducible (typically a few $\fl$'s will suffice). The one dimensional eigenspaces, if any, correspond to the rational Hecke eigenclasses. 

In view of what we explained, a rational Hecke eigenclass $f$ will be regarded, in practice, as an element $\gamma_f\in \Gamma_0^\fD(\fp\fn)$ with the property that for all $\fl\nmid \fp\fn\fD$ one has $[T_\fl([\gamma_f])]= a_\fl [\gamma_f]$ for some $a_\fl\in \Z$, where $T_\fl$ is given by the formula \eqref{eq: T_l}. To lighten the notation, when there is no risk of confusion we will identify $\gamma_f$ with its homology class $[\gamma_f]$; thus we think of $\gamma_f$ as an element of $ H_1(\Gamma_0^\fD(\fp\fn),\Z)$ when convenient.
\subsection{The homology class}\label{subsec: the homology class}
In this subsection we take as input the $\gamma_f\in H_1(\Gamma_0^\fD(\fp\fn),\Z)$ constructed in \S\ref{subsec: the rational eigenclass}, and we provide an algorithmic procedure to compute the homology class $\Delta_f\in H_1(\Gamma,\Div^0\cH_\fp)$ defined in \S\ref{subsec: construction of the lattice}. The first step is to compute the element $c_f\in H_2(\Gamma,\Z)$ which maps to $\gamma_f$ under the map $d\colon H_2(\Gamma,\Z)\ra H_1(\Gamma_0^\fD(\fp\fn),\Z)$ of \eqref{eq: exact sequence in homology}. 

We will again freely use the identification $ G_\text{ab}\simeq H_1(G,\Z)$. Recall that it is induced by the map $\phi$ of \eqref{eq: phi}.  The following properties are straightforward to check:
\begin{enumerate}
 \item $\phi(g_1g_2)=\phi(g_1)+\phi(g_2)-\partial_2(g_1\otimes g_2);$
 \item $\phi([a,b])=\partial_2\left( a\otimes a^{-1} +b\otimes b^{-1} -a\otimes b a^{-1} b^{-1} -b\otimes a^{-1}b^{-1} 
-a^{-1}\otimes b^{-1}+2 \cdot 1_G\otimes 1_G \right)$, where $[a,b]$ denotes the commutator and $1_G$ the identity of $G$.
\end{enumerate}
The second property implies that any element in $[G,G]$ is mapped to a boundary. Such boundary can be effectively computed, as we record in the next lemma.
\begin{lemma}\label{lemma: partial}
 Let $g\in [G,G]$, and suppose that an explicit expression of $g$ as product of commutators is known. Then 
there is an algorithm for explicitly computing  a chain $b\in \Z[G]\otimes\Z[G]$ such that $\phi(g)=\partial_2 (b)$.
\end{lemma}
\begin{proof}
 The algorithm is recursive: write $g=[a,b]\cdot g'$ with $g'$  a product of commutators. If $g'=1$ then we are done 
by property $(2)$ above. If $g'\neq 1$ then by the properties above we have that
 \begin{align*}
  \phi([a,b]\cdot g')=&\phi([a,b])+\phi(g')-\partial_2([a,b]\otimes g')\\
   =& \phi(g')-\partial_2([a,b]\otimes g')\\ +& \partial_2\left( a\otimes a^{-1} +b\otimes b^{-1} -a\otimes b 
a^{-1} b^{-1} -b\otimes a^{-1}b^{-1} 
-a^{-1}\otimes b^{-1}+2 \cdot 1_G\otimes 1_G \right).
 \end{align*}
\end{proof}

Recall the element $\omega_\pi\in R_0^\fD(\fp\fn)^\times$ introduced in \ref{subsec: cohomology groups and hecke operators} when talking about the Atkin--Lehner involutions: it normalizes $\Gamma_0^\fD(\fp\fn)$ and its reduced norm generates $\fp$ (and it is totally positive if $K$ is totally real). We also introduced the notation $\widehat\Gamma_0^\fD(\fn)=\omega_\pi\Gamma_0^\fD(\fn)\omega_\pi^{-1}$. The group $\Gamma_0^\fD(\fp\fn)$ is contained in both $\Gamma_0^\fD(\fn)$ and $\widehat\Gamma_0^\fD(\fn)$, and it is well known that $\Gamma = \Gamma_0^\fD(\fn)\star_{\Gamma_0^\fD(\fp\fn)}\widehat\Gamma_0^\fD(\fn)$, where $\star$ stands for the amalgamated product.

The inclusions $\Gamma_0^\fD(\fp\fn)_\text{ab}\subset\Gamma_0^\fD(\fn)_\text{ab}$ and  $\Gamma_0^\fD(\fp\fn)_\text{ab}\subset\widehat\Gamma_0^\fD(\fn)_\text{ab}$ correspond to the natural homomorphisms
\begin{align*}
 \alpha\colon H_1(\Gamma_0^\fD(\fp\fn),\Z)\lra H_1(\Gamma_0^\fD(\fn),\Z), \ \ \hat\alpha\colon H_1(\Gamma_0^\fD(\fp\fn),\Z)\lra 
H_1(\widehat\Gamma_0^\fD(\fn),\Z).
\end{align*}
The element $\gamma_f\in H_1(\Gamma_0^\fD(\fp\fn),\Z)$ is new at $\fp$. This is equivalent to say that, after extending coefficients to $\Q$, it lies in $\ker(\alpha)\cap \ker(\hat\alpha)$. Therefore, the class of $\gamma_f$ is torsion when viewed as an element in both $\Gamma_0^\fD(\fn)_\text{ab}$ and $\widehat\Gamma_0^\fD(\fn)_\text{ab}$. In particular, there exists $e\in\Z_{>0}$ such that the class of  $\gamma_f^e$ is trivial in $\Gamma_0^\fD(\fn)_\text{ab}$ and $\widehat\Gamma_0^\fD(\fn)_\text{ab}$. Using the algorithms for the word problem of \cite{voight} and \cite{page}, we can find explicit expressions of the form
\begin{align*}
 \gamma_f^e=&\prod [a_i,b_i], \ \ \text{with } a_i,b_i\in\Gamma_0^\fD(\fn);\\
 \gamma_f^e=& \prod [c_j,d_j],\ \ \text{with } c_j,d_j\in\widehat\Gamma_0^\fD(\fn).
\end{align*}
In fact, for computing the second decomposition we can decompose $\omega_\pi\gamma_f^e\omega_\pi$ as a product of commutators  in $\Gamma_0^\fD(\fn)$ and obtain a decomposition in $\widehat\Gamma_0^\fD(\fn)$ by conjugating the found commutators. 

Now, by Lemma \ref{lemma: partial} we can explicitly find elements $z\in \Z[\Gamma_0^\fD(\fn)]\otimes\Z[\Gamma_0^\fD(\fn)]$ and $\hat z \in 
\Z[\widehat\Gamma_0^\fD(\fn)]\otimes\Z[\widehat\Gamma_0^\fD(\fn)]$ such that $\partial_2 z=\gamma_f^e$ and $\partial_2(\hat z)=\gamma_f^e$. Both elements $z$ and $\hat z$ can be viewed naturally as elements in $\Z[\Gamma]\otimes\Z[\Gamma]$, via the inclusions $\Gamma_0^\fD(\fn)\subset\Gamma$ and $\widehat\Gamma_0^\fD(\fn)\subset \Gamma$. Then the element $-z+\hat z\in \Z[\Gamma]\otimes\Z[\Gamma]$ clearly satisfies that $\partial_2 (-z+\hat z)=-\gamma_f^e+\gamma_f^e = 0$, so that it is indeed a $2$-cycle. Its class $c_f$ in $H_2(\Gamma,\Z)$ is the element we were looking for.

By definition $\Delta_f = \delta(c_f)$, so the next step is to compute the image of $c_f$ under the connecting homomorphism $\delta\colon H_2(\Gamma,\Z)\ra H_1(\Gamma,\Div^0\cH_\fp)$ of \eqref{eq: connecting homomorphism in homology}. The following lemma gives an explicit formula in terms of cycles.
\begin{lemma}\label{lemma: delta}
 Let $\tau$ be any element in $\cH_\fp$. The connecting homomorphism $\delta$ is the one induced at the level of chains by the map
\begin{align*}
\begin{array}{ccc}
\Z[\Gamma]\otimes\Z[\Gamma] & \lra & \Z[\Gamma]\otimes\Div^0\cH_\fp \\
g\otimes h & \mapsto & h\otimes (g^{-1}\tau-\tau).
\end{array}
\end{align*}
\end{lemma}
\begin{proof}
 Let $C=\sum n_i g_i\otimes h_i\in \Z[\Gamma]\otimes\Z[\Gamma]$. By the definition of the connecting homomorphism $\delta$ we have 
that 
\begin{align}
 \delta (C)= & \partial_2\left( \sum n_i g_i\otimes h_i\otimes \tau \right)\\
= &\sum n_i h_i\otimes g_i^{-1}\tau-\sum n_i g_i h_i \otimes \tau +\sum n_ig_i\otimes \tau.\label{eq: delta(c)}
 \end{align}
But since $c$ is a cycle we have that $\partial_2\left( \sum n_i g_i \otimes h_i \right)=0$, and therefore
\begin{align*}
 \sum n_i g_i h_i = \sum n_i g_i +\sum n_i h_i.
\end{align*}
From this we have that
\begin{align*}
 \sum n_i g_i h_i \otimes \tau = \sum n_i g_i\otimes\tau +\sum n_i h_i\otimes\tau
\end{align*}
and plugging this into \eqref{eq: delta(c)} we obtain that
\begin{align*}
 \delta(C)=\sum n_i h_i\otimes (g_i^{-1}\tau- \tau).
\end{align*}
\end{proof}

\subsection{The cohomology class} Unlike the homology class of the previous section, the cohomology class $\omega_f$ is exactly the same as that arises in the computation of Darmon points. Explicit algorithms for its calculation were given in \cite{shpquat} in the case where the base field is $K=\Q$, and they can be adapted without much difficulty to general $K$. We next describe the main steps of these algorithms, and the reader is referred to \cite{shpquat} for more details.

The element $\gamma_f\in \Gamma_0^\fD(\fp\fn)$ computed in \ref{subsec: the homology class} gives rise to a cohomology class $\varphi_f\in H^1(\Gamma_0^\fD(\fp\fn),\Z)$. Since the $1$-coboundaries are trivial in this case, there is no necessity of distinguishing between a cocycle and its cohomology class. That is, $\varphi_f$ is just a homomorphism $\Gamma_0^\fD(\fp\fn)_\text{ab}\ra \Z$. We have seen that $\Gamma_0^\fD(\fp\fn)_\text{ab}$ decomposes as a direct sum of irreducible spaces for the action of the Hecke algebra, and that one of the rank $1$ subspaces is generated by $\gamma_f$. Then $\varphi_f$ is the map that sends $\gamma_f$ to $1$ and the elements in the other subspaces to $0$. 

Recall that Shapiro's lemma and \eqref{eq: induced modules} give rise to 
\begin{align}\label{eq: shapiro yet again}
  H^1(\Gamma_0^\fD(\fp\fn),\Z)\simeq H^1(\Gamma,\Ind_{\Gamma_0^\fD(\fp\fn)}^\Gamma\Z)\simeq H^1(\Gamma,\cF(\cE_0,\Z)).
\end{align}
When constructing $\omega_f$ we saw that the image of $\varphi_f$ on the group of the right lies in the image of the natural map \begin{align*}H^1(\Gamma,\text{HC}(\Z))\stackrel{\rho}{\ra}H^1(\Gamma,\cF(\cE_0,\Z)),\end{align*} and a preimage is, by definition, $\omega_f$. The isomorphisms \eqref{eq: shapiro yet again} are induced by maps on cocycles which are completely explicit, so one can effectively compute a cocycle in $Z^1(\Gamma,\cF(\cE_0,\Z))$ whose class corresponds to the image of $\varphi_f$. However, this cocycle \emph{will not} in general take values in the submodule $\text{HC}(\Z)$ of $\cF_0(\cE,\Z)$ (what it is true is that it will be \emph{cohomologous} to a cocycle with values in harmonic cocycles). 

The problem is that the map on cocycles that induces Shapiro's isomorphism depends on the choice of a system of representatives for $\Gamma_0^\fD(\fp\fn)\backslash\Gamma$. Different choices lead to different cocycles in $Z^1(\Gamma,\cF(\cE_0,\Z))$. Of course, all of them are cohomologous inside $Z^1(\Gamma,\cF(\cE_0,\Z))$, but only some of them actually lie in $Z^1(\Gamma,\text{HC}(\Z))$. Following an idea introduced in \cite[\S 4]{LRV}, it is possible to choose a system of representatives in such a way that the obtained cocycle directly takes values in $\text{HC}(\Z)$. They are called \emph{radial systems}, and we next recall their definition.

Let us denote by $\Z_\fp$ the completion of $\cO_K$ at $\fp$ and by $\F_\fp$ its residue field. The first step is to compute a system of representatives $\Upsilon=\{\gamma_a\}_{a\in \P^1(\F_\fp)}$ for $\Gamma_0^\fD(\fp\fn)\backslash \Gamma_0^\fD(\fn)$ satisfying that:
\begin{align*}
\gamma_\infty = 1,\ \text{ and } \ \iota_\fp(\gamma_a) = u_a \smtx{0}{-1}{1}{\tilde a}\ \text{ (here  $\tilde a $  is  a lift of $a$ to $\Z_\fp$)}, 
\end{align*}
where $u_a$ belongs to
\begin{align*}
  \Gamma_0^{\fD,\text{loc}}(\fp) = \{ \smtx a b c d \in \SL_2(\Z_\fp) \mid c\in\fp \}.
\end{align*}
This induces a system of representatives $\{\tilde \gamma_a\}_{a\in\P^1(\F_\fp)}$ of $\Gamma_0^\fD(\fp\fn)\backslash \widehat \Gamma_0^\fD(\fp\fn)$ by putting $\tilde\gamma_\infty= 1$ and $\tilde\gamma_a = \pi^{-1}\omega_\pi\gamma_a\omega_\pi$ for $a\neq \infty$.

We will index the representatives of $\Gamma_0^\fD(\fp\fn)\backslash \Gamma$ by edges in $\cE_0$ (recall that these are in bijection, cf. \eqref{eq: correspondences}), and the representatives of $\Gamma_0^\fD(\fn)\backslash \Gamma $ by vertices in $ \cV_0$. We define $\{\gamma_e\}_{e\in\cE_0}$ and $\{\gamma_v\}_{v\in\cV}$ to be the systems of representatives uniquely determined by the conditions:
\begin{itemize}
\item $\gamma_{v_0}=\gamma_{v_1}=1$;
 \item $\{\gamma_e\}_{s(e)=v}=\{\gamma_a\gamma_v\}_{a\in \P^1(\F_\fp)}$ for all $v\in\cV_0$;
\item  $\{\gamma_e\}_{t(e)=v}=\{\tilde\gamma_a\gamma_v\}_{a\in \P^1(\F_\fp)}$ for all $v\in\cV_1$;
\item  $\gamma_{s(e)} = \gamma_e$ for all $e\in\cE_0$ such that $d(t(e),v_0)< d(s(e),v_0)$;
\item  $\gamma_{t(e)} = \gamma_e$ for all $e\in\cE_0$ such that $d(t(e),v_0)>d(s(e),v_0)$.
\end{itemize}
We next describe a cocycle $\mu_f$ which represents the image of $\varphi_f$ under \eqref{eq: shapiro yet again}. In order to lighten the notation we set $\mu=\mu_f$, since $f$ is fixed in this discussion. For $e\in\cE_0$ and $g\in \Gamma$, let $h(g,e)\in \Gamma_0^\fD(\fp\fn)$ be the element determined by the identity
\begin{align}\label{eq: def of h}
\gamma_eg=h(g,e)\gamma_{g^{-1}e}.
\end{align}
 Now for $g\in \Gamma$, let $\mu_{g}\colon \cF(\cE_0,\Z)\ra $ be the map defined by
\begin{align}\label{eq: def of mu} 
 \mu_{g}(e)=\varphi_f({h(g,e)}), \ \text{for $e\in \cE_0$}.
\end{align}
Since the system of representatives of $\Gamma_0^\fD(\fp\fn)\backslash \Gamma$ was taken to be radial,  $\mu_g$ belongs in fact to $\text{HC}(\Z)$ (cf. \cite[Proposition 4.8]{LRV}). In addition, $\mu$ is a $1$-cocycle, i.e.,  $\mu\in Z^1(\Gamma,\text{HC}(\Z))$. 

The cocycle $\mu$ is not yet a cocycle representing the cohomology class $\omega_f$, but almost. The last step is to ``project to the cuspidal part''. For this, let $\fl$ be a prime not dividing $\fp\fn\fD$ and consider the projector $T_\fl- |\fl| -1$. The cocycle $(T_\fl-|\fl|-1)\mu$ turns to be the correct one, i.e., it represents (a multiple of) $\omega_f$. Since considering a multiple of $\omega_f$ does not change the homothety class of the lattice $\Lambda_f$, we can assume that $\omega_f$ is given by $(T_\fl-|\fl|-1)\mu$.

In view of the above discussion, the calculation of $\omega_f$ in practice boils down to the effective computation of the elements $h(g,e)$ of \eqref{eq: def of h}. This can be done with the algorithm of \cite[Theorem 4.1]{shpquat}. 

\subsection{The multiplicative pairing} In order to simplify a little bit the computations, it is convenient to use the Hecke equivariance of the integration pairing and write
\begin{align*}
q_f=  \Xint\times\langle \omega_f,\Delta_f \rangle = \Xint\times\langle (T_\fl-|\fl|-1)\mu,\Delta_f \rangle = \Xint\times\langle \mu,(T_\fl-|\fl|-1)\Delta_f \rangle, 
\end{align*}
where $\mu$ is the explicit cocycle defined in \eqref{eq: def of mu}. The reason is that the Hecke action is slightly easier to compute on $H_1(\Gamma_0^\fD(\fp\fn),\Div^0\cH_\fp)$ than on $H^1(\Gamma_0^\fD(\fp\fn),\text{HC}(\Z))$, simply because the coefficients are easier to manipulate. Indeed, we use the explicit formula \eqref{eq: T_l} to compute $T_\fl\Delta_f$. 

Now $(T_\fl-|\fl|-1)\Delta_f$ is of the form
\begin{align*}
  (T_\fl-|\fl|-1)\Delta_f = \sum g_i\otimes (\tau_i'-\tau_i), \ \text{ for certain } g_i\in \Gamma \text{ and } \tau_i',\tau_i\in\cH_\fp,
\end{align*}
so that 
\begin{align*}
  q_f = \prod \Xint\times_{\P^1(K_\fp)} \left( \frac{t-\tau_i'}{t-\tau_i}\right)d\mu_g(t).
\end{align*}
Therefore, computing $q_f$ boils down to compute multiplicative integrals of the form
\begin{align}\label{eq: integrals}
  \Xint\times_{\P^1(K_\fp)} \left(\frac{t-\tau_2}{t-\tau_1} \right)d\mu_g(t).
\end{align}
These integrals can in principle be computed, up to finite precision, by Riemann products. Namely, by taking a finite covering $\cU$ of $\P^1(K_\fp)$ and evaluating the expression appearing in \eqref{eq: mult int}. However, this method is of exponential complexity in terms of the number of $p$-adic digits of accuracy, and it is too inefficient for practical purposes. 

Integrals \eqref{eq: integrals} can be computed instead by using the method of overconvergent cohomology of \cite{pollack-pollack}, a generalization of Steven's overconvergent modular symbols (cf. \cite{PS}) which is of polynomial complexity and much more efficient in practice. This method is explained in \cite[\S 5]{shpquat} for the case where $K=\Q$. However, the assumption that $K=\Q$ is by no means essential, and all the calculations and algorithms of loc. cit. go through with no essential difficulty to any $K$.

\section{Numerical computations}\label{sec: numerical computations and tables}
We have implemented in Sage \cite{sage} the algorithms described in Section \ref{sec: explicit methods and algorithms} that compute approximations to $q_f$. Some of the code uses routines that are currently only available in Magma. The implementation is done under the (inessential) additional restriction that the prime $\fp$ is of residual degree $1$. This simplifies the routines involving calculations in the local field, since in that case $K_\fp$ is $\Q_p$  (here $p$ is the norm of $\fp$), rather than an extension of $\Q_p$. The code and the instructions for using it are available at \url{https://github.com/mmasdeu/darmonpoints}.

\subsection{Recovering the curve from the $\mathcal{L}$-invariant}
Recall that the rational Hecke eigenclass $f$ of level $\fp\fn$ on a quaternion algebra of discriminant $\fD$ should correspond to an elliptic curve $E_f$ of conductor $\fN= \fp\fD\fn$. According to Conjecture \ref{conj: our main}, the lattice generated by $q_f$ is commensurable with the Tate lattice of a curve satisfying the defining properties of $E_f$ (i.e., a curve of conductor $\fN$ and such that $\# E_f(\cO_K/\fl)= |\fl| + 1 -a_\fl(f)$ for all primes $\fl$ of $K$). In order to test this conjecture we use the calculated $q_f$ to ``discover'' an equation for $E_f$.

Roughly speaking, the idea is that conjecturally the Tate parameter of $E_f$ is of the form $q_f^r$ for  some $r\in\Q$, and from the Tate parameter one can compute the $j$-invariant by a well-known power series. Therefore, the problem reduces to that of computing the equation of a curve over $K$, given a $p$-adic approximation to its $j$-invariant. For this we use some of the methods of \cite{cremona-lingham}. 

More precisely, we look for a Weierstrass model of the form
\begin{align*}
 y^2 = x^3 - \frac{c_4}{48} x - \frac{c_6}{864},
\end{align*}
for which we know (an approximation of) the $j$-invariant. The main idea is to use the relation  $ j = {c_4^3}/{\Delta}$, where $\Delta$ denotes the discriminant of the above model. Of course we do not know $\Delta$ a priori, but we have certain control on it: by \cite[Proposition 3.2]{cremona-lingham}, its class in $K^\times/(K^\times)^{12}$ belongs to 
\begin{align*}
  K(\fN,12) = \{x\in K^\times/(K^\times)^{12} \mid v_\fq(\fN)\equiv 0 \pmod{12} \text{ for all primes } \fq \mid \fN\},
\end{align*}
which is a finite set. What we do is to run over $\Delta$'s in $K(\fN,12)$; for each try of $\Delta$, we can assume that the valuation of $\Delta$ equals the valuation of $q_E$, and from this we get the $r$ for which $q_f^r$ is a candidate for $q_E$. We then compute $j$ from the candidate to $q_E$ and try to recognize $\sqrt[3]{j\Delta}$ as an element of $K$. If we succeed, this is the $c_4$ and the $c_6$ can then be computed by means of $\Delta = (c_4^3-c_6^2)/1728$.

Summing up, the algorithm that we use is the following. The input is the element $q_f\in \C_p^\times$, which we have computed up to, say, $N$ digits of $p$-adic accuracy (we can assume that $v_p(q_f)>0$, for we can replace $q_f$ by $q_f^{-1}$). In all the examples we have tried, $q_f$ turns out to lie in $\Q_p^\times$. This is of course consistent with Conjecture \ref{conj: our main}, because the Tate period of $E_f$ lies in $K_\fp^\times=\Q_p^\times$.

\begin{enumerate}
\item Set $d:=v_\fp(q_f)$, and compute the finite number of elements $q_0$ such that $q_0^d=q_f$ (in particular, $v_\fp(q_0)=1$). 
\item For every $q_0$ as above, run over the finite number of $\Delta\in K(\fN,12)$ and set $q=q_0^{v_p(\Delta)}$. (This is the candidate for $q_E$ at this step.) 
\item For each $q$ as above, compute $j = j(q)$ by means of the power series $j(q) = 1/q + 744+ 196884q+\cdots$. This gives an element $j\in\Q_p^\times$, known up to precision $p^N$. Then compute $c_4'=\sqrt[3]{j\Delta}\in \Q_p^\times$.
\item Using standard recognition techniques, try to find $c_4\in K$ which coincides with $c_4'$ up to precision $p^N$. If such a $c_4$ is found, test whether $c_4^3-1728\Delta$ is a square in $K$ and, if so, set $c_6$ as one of its square roots.
\item If in the previous step we have found $c_4,c_6\in K$, compute the conductor of the curve $y^2 = x^3-\frac{c_4}{48}x-\frac{c_6}{864}$. If the conductor is equal to $\fN$, then return this curve. Note that if we reach this step, then there are six possilibities to try, for $c_4$ can be modified by third roots of unity, and $c_6$ by a sign.
\end{enumerate}
Two remarks are in order here:
  \begin{enumerate}[a)]
\item Observe that the precision to which we need to know $q_f$ is determined by the height of the $c_4$ in a Weierstrass model of $E_f$. Indeed, if the precision of $q_f$ is too small one is in general not able to recognize $c_4$ from its $p$-adic approximation $c_4'$.
  \item If the above algorithm returns the equation of a curve, Conjecture \ref{conj: our main} would imply that it is an equation of $E_f$. In that sense, one might think that the algorithm is only conjectural. However, if it returns a curve one can always check \emph{a posteriori} whether such a curve satisfies the defining properties of $E_f$, by checking that its $a_\fp$'s coincide with the eigenvalues of $f$ by $T_\fp$.
  \end{enumerate}

\subsection{Numerical results} We have performed systematic calculations for totally real fields of degree $2$ and $3$, and for ATR fields of degree $2$, $3$, and $4$. For each of these degrees, we have considered the number fields of narrow class number $1$ and discriminant in absolute value up to $5000$ (this data was obtained from LMFDB \cite{lmfdb}). For each such number field $K$ we have exhausted all levels $\fN$ up to a norm $200$ which satisfy certain additional restrictions. First of all, recall that the method presented in this note can only be applied to those $\fN$ satisfying that:
\begin{itemize}
\item $\fN$ can be factored into pairwise coprime ideals $\fp\fD\fn$, where $\fp$ is prime and $\fD$ is the discriminant of a quaternion algebra $B/K$ which is split at one archimedean place.
\end{itemize}
In addition, we have imposed additional restrictions in order to simplify the coding of some routines and speed up the computations. Namely:
\begin{itemize}
\item The norm of $\fp$ is at most $23$ (primes of large norm slow down our implementation of the integration routines);
\item The norm of $\fp$ is a prime number (so that $K_\fp\simeq\Q_p$ rather than a finite extension, which simplifies the $p$-adic routines).
\end{itemize}
For every field $K$ and every factorization of $\fN=\fp\fD\fn$ satisfying the conditions above, we have computed $H_1(\Gamma_0^\fD(\fp\fn),\Q)$ (for a choice of the quaternion algebra $B$ of discriminant $\fD$ and that splits at one archimedean place). For most of the levels this homology group does not contain any rational Hecke eigenline, and thus one does not expect an elliptic curve of that conductor. For the levels in which there are rational lines, we have computed the $\mathfrak{L}$-invariant of each line, and tried to recognize an algebraic curve over $K$ whose $\mathfrak{L}$-invariant matches up to high $p$-adic precision and whose conductor is $\fN$.

In the appendices we provide tables for the results of these computations. Each row contains the number field $K$, the level $\fN$ factored as $\fp\fD\fn$ and the coefficients $c_4$ and $c_6$ for the found curve of conductor $\fN$. These $c_4$ and $c_6$ are not necessarily minimal, in the sense that there might be curves of smaller height in the same isogeny class.

We warn the reader that these tables are not complete in the sense that for each $K$ not necessarily all the levels $\fN$ of norm $\leq 200$ and satisfying the above restrictions appear. The first reason is that $H_1(\Gamma_0^\fD(\fN),\Q)$ might not contain any rational line and no curve is expected at that level. Such levels can also be of some interest and they can be found in a more complete version of the tables at \url{https://github.com/mmasdeu/elliptic_curve_tables}. Another reason, this one related to our implementation, is that we imposed a limitation of time and computations taking too long were stopped\footnote{We limited to $30$ minutes the time allowed to compute the arithmetic group $\Gamma_0^\fD(\fp\fn)$, and to $120$ minutes the time to do the rest of the calculation (homology class, cohomology class, and integration pairing)}. Also, in some occasions, the $\fp$-adic lattice has been successfully computed, but we have not been able to recognize an algebraic curve of the right conductor from the Tate period $q_f$. This usually happens when the precision to which we have computed $q_f$ (which is roughly $100$ decimal digits, in our case) is not enough because the curve has too large height. Finally, runtime errors have occasionally arisen.

We remark that for each $\fN$ there might be several choices for the prime $\fp$, as well as several choices for the factorization of $\fN$ as $\fp\fD\fn$. In particular, in the tables it is sometimes the case that the same (isogeny class of) elliptic curve is found from different factorizations of $\fN$.

\section{Discussion and further improvements}
The extensive numerical calculations that we have carried out provide some evidence of the validity of Conjecture \ref{conj: our main}. They also illustrate how the construction of the $\fp$-adic lattice can be translated into explicit algorithms which are well suited for systematic computations. 

Along the text we imposed a number of conditions to the fields and levels that we consider. Some of these conditions are inherent to the method;  the main one is the necessity of having a prime $\fp\mid\mid \fN$ and a factorization $\fN=\fp\fD\fn$ with $\fD$ the discriminant of a quaternion algebra over $K$ that splits at one archimedean place. Most of the other extra restrictions we imposed were just simplifying assumptions. Therefore, it might be interesting to relax them, as that would enlarge the types of fields and levels for which one is able to compute curves. Some of the possible improvements, both to the given algorithms and to our current implementation of them, might be:
\begin{itemize}
\item Do the local computations over finite extensions of $\Q_p$; this would allow to treat $\fp$'s of residual degree $>1$. 
\item Improve the integration routines in order to allow $\fp$'s of higher norm.
\item One of the bottlenecks of our current implementation is the computation of $\Gamma_0^\fD(\fp\fn)\subset B$ using the routines of John Voight and Aurel Page. This is usually much more computationally demanding than computing $\Gamma_0^\fD(1)$, the norm one elements of  a maximal order. In this kind of situations, a usual trick is to replace groups of the form $H_i(\Gamma_0^\fD(\fp\fn),A)$ by $H_i(\Gamma_0^\fD(1),\Ind_{\Gamma_0^\fD(1)}^{\Gamma_0^\fD(\fp\fn)}A)$ via Shapiro's Lemma. 
Implementing this approach is likely to lead to an improvement of the overall running time.
\item Develop algorithms for working with (co)homology groups of degree higher than one. This would allow to treat fields $K$ having more than one complex place.
\item Provide a construction of the lattice $\Lambda_f$ when $K$ has narrow class number $>1$. This would probably involve working adelically.
\end{itemize}

\bibliographystyle{halpha}
\bibliography{refs}
\newpage

\appendix
\section{Tables}
We include tables of curves for number fields other than $\QQ$ of signatures $(r,s)$ with $s\leq 1$, and for which $r+s\leq 3$. That is, when $s=0$ we looked at totally real quadratic and cubic fields; when $s=1$ we looked at cubic ATR fields (of signature $(1,1)$) and quartic ATR fields (of signature $(2,1)$). Each row of the tables consists of five columns:
\begin{enumerate}
\item the absolute value $|\Delta_K|$ of the discriminant of the field $K$,
\item The coefficients $[b_0,\ldots,b_{n-1}]$ of a minimal polynomial $f_K(x)=x^n+b_{n-1}x^{n-1}+\cdots b_1x+b_0$ of $K$.
\item The norm $\operatorname{Nm}(\fN)$ of an ideal $\fN$ (the level).
\item A factorization $\fN=\fp\fD\fm$ of the the level. All ideals are principal, and we use the notation $(\alpha)_a$ to indicate the ideal generated by an element $\alpha\in\cO_K$ of norm $a$.
\item The coefficients $c_4(E)$ and $c_6(E)$ of the elliptic curve $E$ expressed in terms of $r$, a root of $f_K(x)$.
\end{enumerate}

\section*{Real quadratic fields}

\begin{longtable}[H]{lrlr}
Found curves & $107$ & Time outs & $507$ \tabularnewline
No rational lines & $19$ & Not recognized & $38$ \tabularnewline
Total Trials & $743$ & $\fp$ too large & $6$ \tabularnewline
 &  & Runtime errors & $66$ \tabularnewline
\end{longtable}
{\tiny
\begin{longtable}[H]{rp{1.3cm}rlp{7cm}}
$|\Delta_K|$ & $f_K(x)$ & $\text{Nm}(\mathfrak{N})$ & $\mathfrak{p}\mathfrak{D}\mathfrak{m}$ & $c_4(E),c_6(E)$ \tabularnewline \hline\endhead
$5$ & $\left[-1, -1\right]$ & $55$ & $(-3 r + 1)_{11} (-2 r + 1)_{5} (1)$ & $1224 r + 737,$ \tabularnewline
 &  &  &  & $62972 r + 39767$ \tabularnewline\hline
$5$ & $\left[-1, -1\right]$ & $55$ & $(-3 r + 2)_{11} (-2 r + 1)_{5} (1)$ & $-1224 r + 1961,$ \tabularnewline
 &  &  &  & $-62972 r + 102739$ \tabularnewline\hline
$5$ & $\left[-1, -1\right]$ & $55$ & $(-3 r + 2)_{11} (-2 r + 1)_{5} (1)$ & $-24 r + 41,$ \tabularnewline
 &  &  &  & $28 r - 77$ \tabularnewline\hline
$5$ & $\left[-1, -1\right]$ & $76$ & $(4 r - 3)_{19} (2)_{4} (1)$ & $1368 r - 143,$ \tabularnewline
 &  &  &  & $21924 r + 58751$ \tabularnewline
\end{longtable}
}

\section*{Cubic totally real fields}

\begin{longtable}[H]{lrlr}
Found curves & $23$ & Time outs & $1349$ \tabularnewline
No rational lines & $71$ & Not recognized & $19$ \tabularnewline
Total Trials & $1630$ & $\fp$ too large & $23$ \tabularnewline
 &  & Runtime errors & $145$ \tabularnewline
\end{longtable}
{\tiny
\begin{longtable}[H]{rp{1.3cm}rlp{7cm}}
$|\Delta_K|$ & $f_K(x)$ & $\text{Nm}(\mathfrak{N})$ & $\mathfrak{p}\mathfrak{D}\mathfrak{m}$ & $c_4(E),c_6(E)$ \tabularnewline \hline\endhead
$148$ & $\left[1, -3, -1\right]$ & $17$ & $(2 r + 1)_{17} (1) (1)$ & $269381372 r^{2} + 315199506 r - 124133929,$ \tabularnewline
 &  &  &  & $8968858286035 r^{2} + 10494339880946 r - 4132949696398$ \tabularnewline\hline
$148$ & $\left[1, -3, -1\right]$ & $17$ & $(2 r + 1)_{17} (1) (1)$ & $269381372 r^{2} + 315199506 r - 124133929,$ \tabularnewline
 &  &  &  & $8968858286035 r^{2} + 10494339880946 r - 4132949696398$ \tabularnewline\hline
$148$ & $\left[1, -3, -1\right]$ & $19$ & $(-r^{2} - r - 1)_{19} (1) (1)$ & $76607638 r^{2} + 89637562 r - 35301652,$ \tabularnewline
 &  &  &  & $-1730949724156 r^{2} - 2025360881264 r + 797640893496$ \tabularnewline\hline
$148$ & $\left[1, -3, -1\right]$ & $19$ & $(-r^{2} - r - 1)_{19} (1) (1)$ & $80 r^{2} + 96 r - 32,$ \tabularnewline
 &  &  &  & $-39936 r^{2} - 46736 r + 18392$ \tabularnewline
\end{longtable}
}

\section*{Imaginary quadratic fields}

\begin{longtable}[H]{lrlr}
Found curves & $10$ & Time outs & $30$ \tabularnewline
No rational lines & $168$ & Not recognized & $6$ \tabularnewline
Total Trials & $218$ & $\fp$ too large & $0$ \tabularnewline
 &  & Runtime errors & $4$ \tabularnewline
\end{longtable}
{\tiny
\begin{longtable}[H]{rp{1.3cm}rlp{7cm}}
$|\Delta_K|$ & $f_K(x)$ & $\text{Nm}(\mathfrak{N})$ & $\mathfrak{p}\mathfrak{D}\mathfrak{m}$ & $c_4(E),c_6(E)$ \tabularnewline \hline\endhead
$3$ & $\left[1, -1\right]$ & $196$ & $(3 r - 2)_{7} (-6 r + 2)_{28} (1)$ & $-131065 r,$ \tabularnewline
 &  &  &  & $47449331$ \tabularnewline\hline
$3$ & $\left[1, -1\right]$ & $196$ & $(-3 r + 1)_{7} (6 r - 4)_{28} (1)$ & $-131065 r,$ \tabularnewline
 &  &  &  & $47449331$ \tabularnewline\hline
$4$ & $\left[1, 0\right]$ & $130$ & $(3 r - 2)_{13} (-r - 3)_{10} (1)$ & $-264 r + 257,$ \tabularnewline
 &  &  &  & $-6580 r + 2583$ \tabularnewline\hline
$4$ & $\left[1, 0\right]$ & $130$ & $(-3 r - 2)_{13} (-3 r - 1)_{10} (1)$ & $264 r + 257,$ \tabularnewline
 &  &  &  & $6580 r + 2583$ \tabularnewline\hline
$7$ & $\left[2, -1\right]$ & $44$ & $(r)_{2} (3 r + 1)_{22} (1)$ & $648 r + 481,$ \tabularnewline
 &  &  &  & $-28836 r + 4447$ \tabularnewline\hline
$7$ & $\left[2, -1\right]$ & $44$ & $(r - 1)_{2} (3 r - 4)_{22} (1)$ & $-648 r + 1129,$ \tabularnewline
 &  &  &  & $28836 r - 24389$ \tabularnewline\hline
$8$ & $\left[2, 0\right]$ & $99$ & $(r + 1)_{3} (-4 r + 1)_{33} (1)$ & $444 r + 25,$ \tabularnewline
 &  &  &  & $14794 r - 16263$ \tabularnewline\hline
$8$ & $\left[2, 0\right]$ & $99$ & $(r - 1)_{3} (-4 r - 1)_{33} (1)$ & $-444 r + 25,$ \tabularnewline
 &  &  &  & $-14794 r - 16263$ \tabularnewline\hline
$8$ & $\left[2, 0\right]$ & $99$ & $(-r - 3)_{11} (3)_{9} (1)$ & $-444 r + 25,$ \tabularnewline
 &  &  &  & $-14794 r - 16263$ \tabularnewline\hline
$8$ & $\left[2, 0\right]$ & $99$ & $(r - 3)_{11} (3)_{9} (1)$ & $444 r + 25,$ \tabularnewline
 &  &  &  & $14794 r - 16263$ \tabularnewline\hline
\caption{Fields of degree 2}\\
\end{longtable}
}

\section*{Cubic ATR fields}

\begin{longtable}[H]{lrlr}
Found curves & $130$ & Time outs & $5199$ \tabularnewline
No rational lines & $1527$ & Not recognized & $35$ \tabularnewline
Total Trials & $8019$ & $\fp$ too large & $9$ \tabularnewline
 &  & Runtime errors & $1119$ \tabularnewline
\end{longtable}
{\tiny
\begin{longtable}[H]{rp{1.3cm}rlp{7cm}}
$|\Delta_K|$ & $f_K(x)$ & $\text{Nm}(\mathfrak{N})$ & $\mathfrak{p}\mathfrak{D}\mathfrak{m}$ & $c_4(E),c_6(E)$ \tabularnewline \hline\endhead
$23$ & $\left[1, 0, -1\right]$ & $185$ & $(r^{2} + 1)_{5} (3 r^{2} - r + 1)_{37} (1)$ & $643318 r^{2} - 1128871 r + 852306,$ \tabularnewline
 &  &  &  & $925824936 r^{2} - 1624710823 r + 1226456111$ \tabularnewline\hline
$31$ & $\left[-1, 1, 0\right]$ & $129$ & $(-r - 1)_{3} (-3 r^{2} - 2 r - 1)_{43} (1)$ & $-4787 r^{2} + 10585 r + 3349,$ \tabularnewline
 &  &  &  & $1268769 r^{2} - 371369 r + 424764$ \tabularnewline\hline
$44$ & $\left[1, 1, -1\right]$ & $121$ & $(2 r - 1)_{11} (r^{2} + 2)_{11} (1)$ & $4097022 r^{2} - 6265306 r + 7487000,$ \tabularnewline
 &  &  &  & $14168359144 r^{2} - 21861492432 r + 26039140708$ \tabularnewline\hline
$44$ & $\left[1, 1, -1\right]$ & $121$ & $(2 r - 1)_{11} (r^{2} + 2)_{11} (1)$ & $1774 r^{2} - 1434 r - 1304,$ \tabularnewline
 &  &  &  & $-42728 r^{2} - 123104 r - 54300$ \tabularnewline
\end{longtable}
}

\section*{Quartic ATR fields}

\begin{longtable}[H]{lrlr}
Found curves & $94$ & Time outs & $1378$ \tabularnewline
No rational lines & $2860$ & Not recognized & $1$ \tabularnewline
Total Trials & $4640$ & $\fp$ too large & $122$ \tabularnewline
 &  & Runtime errors & $185$ \tabularnewline
\end{longtable}
{\tiny
\begin{longtable}[H]{rp{1.3cm}rlp{7cm}}
$|\Delta_K|$ & $f_K(x)$ & $\text{Nm}(\mathfrak{N})$ & $\mathfrak{p}\mathfrak{D}\mathfrak{m}$ & $c_4(E),c_6(E)$ \tabularnewline \hline\endhead
$643$ & $\left[1, -2, 0, -1\right]$ & $175$ & $(r^{3} - r^{2} - r - 1)_{7} (2 r^{3} - r^{2} - 2)_{25} (1)$ & $-1783 r^{3} + 1032 r^{2} + 522 r + 3831,$ \tabularnewline
 &  &  &  & $116369 r^{3} - 62909 r^{2} - 30125 r - 248439$ \tabularnewline\hline
$688$ & $\left[-1, -2, 0, 0\right]$ & $11$ & $(-r^{3} + r^{2} + r + 2)_{11} (1) (1)$ & $200 r^{3} + 284 r^{2} + 376 r + 136,$ \tabularnewline
 &  &  &  & $-5184 r^{3} - 7280 r^{2} - 10024 r - 3672$ \tabularnewline\hline
$688$ & $\left[-1, -2, 0, 0\right]$ & $19$ & $(2 r^{3} - 3)_{19} (1) (1)$ & $552 r^{3} + 764 r^{2} + 1064 r + 392,$ \tabularnewline
 &  &  &  & $-11536 r^{3} - 16160 r^{2} - 22584 r - 8312$ \tabularnewline\hline
$731$ & $\left[-1, 0, 2, -1\right]$ & $80$ & $(r^{2} + 1)_{5} (1) (2)_{16}$ & $-848 r^{3} + 1529 r^{2} + 456 r - 420,$ \tabularnewline
 &  &  &  & $45471 r^{3} - 164824 r^{2} + 11648 r + 72230$ \tabularnewline
\end{longtable}
}

\end{document}